\documentclass[10pt,reqno]{amsart}
 
\usepackage[english]{babel}
\usepackage{dsfont} 
\usepackage{graphicx}
\usepackage{color}
\usepackage{hyperref}
\usepackage{cleveref}
\usepackage{enumerate}
\usepackage{amssymb,empheq,mathrsfs} 
\usepackage{amsthm} 
\usepackage{xcolor}
\usepackage[foot]{amsaddr}
\usepackage[normalem]{ulem}

\usepackage{moreverb}

\allowdisplaybreaks

\newtheorem{thm}{Theorem}[section]
\newtheorem{lem}[thm]{Lemma}

\newtheorem{cor}[thm]{Corollary}

\theoremstyle{definition}

\newtheorem{dfn}[thm]{Definition}

\newtheorem{exm}[thm]{Example}

\theoremstyle{remark}
\newtheorem{rem}[thm]{Remark}

\newcommand{\exmsymbol}{\hfill$\circ$}

\newcommand{\cset}{\mathds{C}}

\newcommand{\nset}{\mathds{N}}

\newcommand{\rset}{\mathds{R}}

\newcommand{\zset}{\mathds{Z}}

\newcommand{\diff}{\mathrm{d}}

\newcommand{\supp}{\mathrm{supp}\,}

\newcommand{\bimap}{\hookrightarrow\mathrel{\mspace{-15mu}}\rightarrow}

\newcommand{\cB}{\mathcal{B}}
\newcommand{\cat}{\mathcal{C}}

\newcommand{\cH}{\mathcal{H}}

\newcommand{\cS}{\mathcal{S}}

\newcommand{\cZ}{\mathcal{Z}}

\newcommand{\fC}{\mathfrak{C}}

\newcommand{\fd}{\mathfrak{d}}

\newcommand{\fI}{\mathfrak{I}}

\newcommand{\fp}{\mathfrak{p}}

\newcommand{\ft}{\mathfrak{t}}
\newcommand{\fT}{\mathfrak{T}}

\definecolor{DarkBlue}{rgb}{0,0.2,0.6}
\definecolor{PinkPurple}{rgb}{0.8,0.3,0.3}
\definecolor{darkgreen}{rgb}{.1,.5,0}
\definecolor{brown}{rgb}{.4,.2,.1}

\author{Ra\'ul E. Curto}
\address{University of Iowa, Department of Mathematics, Iowa City, Iowa 52246, U.S.A.}
\email{raul-curto@uiowa.edu}

\author{Philipp J.\ di~Dio}
\address{Universit\"at Leipzig, Institut f\"ur Mathematik und Informatik, Augustplatz 10, D-04109 Leipzig, Germany}
\email{didio@uni-leipzig.de}



\begin{title}[Time-dependent moments from the heat equation]{Time-dependent moments from the heat equation and a transport equation}
\end{title}

\begin{document}

\begin{abstract}
We present a new connection between the classical theory of full and truncated moment problems and the theory of partial differential equations, as follows. \ For the classical heat equation $\partial_t u = \nu \Delta u$, with initial data $u_0 \in\cS(\rset^n)$, we first compute the moments $s_{\alpha}(t)$ of the unique solution $u \in \cS(\rset^n)$. \ These moments are polynomials in the time variable, of degree comparable to $\alpha$, and with coefficients satisfying a recursive relation. \ This allows us to define the polynomials for any sequence, and prove that they preserve some of the features of the heat kernel. \ In the case of moment sequences, the polynomials trace a curve (which we call the heat curve) which remains in the moment cone for positive time, but may wander outside the moment cone for negative time. \ This provides a description of the boundary points of the moment cone which are also moment sequences. \ We also study how the determinacy of a moment sequence behaves along the heat curve. \ Next, we consider the transport equation $\partial_t u = ax \cdot \nabla u$, and conduct a similar analysis. \ Along the way we incorporate several illustrating examples.
\end{abstract}

\maketitle

\tableofcontents

\section{Introduction}

Let $\mu$ be a positive Borel measure on $\rset^n$, such that $\mathbb{R}[x_1,\ldots,x_n] \subseteq L^1(\mu)$. \ The moment $s_\alpha$ with $\alpha\in\nset_0^n$ of $\mu$ is
\begin{equation}\label{eq:moment}
s_\alpha := \int_{\rset^n} x^\alpha~\diff\mu(x).
\end{equation}
The moment problem entails the question: \ Given a sequence $(s_\alpha)_{\alpha\in\nset_0^n}$ of real numbers, does there exist a measure $\mu$ such that (\ref{eq:moment}) holds?

As simple as the formulation of this question is, finding an answer (and the theory behind it) is an enormous undertaking, only exceeded by the applications in optimization, integral approximation, non-negative polynomials, statistics, shape reconstruction, and other fields, see e.g.\ the monographs and reviews \cite{shohat43,ahiezer62,akhiezClassical,kreinMarkovMomentProblem,lauren09,
lasserreSemiAlgOpt,schmudMomentBook} and the research articles 
\cite{havila35,havila36,shohat43,richte57,krein70,kemper71,stroud71,
curto3,curto05,curto13,vasile14,gravin14,stoyan16,fialkow17,didio17owr,
infusino17,didio17Cara,gravin18,rienerOptima,didio18gaussian,marx20,didio21HilbertFunction}, to name only a few.

Instead of focusing our study on one fixed sequence $(s_\alpha)_\alpha$, in this paper we are interested in a $1$--parameter family of sequences; that is, a family depending on a parameter $t$ we call time: $(s_\alpha(t))_\alpha$. \ In order to gain insightful information, we ask that the representing measure $\mu$ of the moment sequence fulfills an evolution equation. \ We study here two simple cases: the heat equation
$\partial_t u = \nu\Delta u$
and the transport equation
$\partial_t u = x\nabla u$.
The representing measure will be necessarily of the form $\diff \mu_t(x) = u(x,t)~\diff x$. First, we calculate the time-dependent moments for these cases. \ We then investigate, starting with a moment sequence at $t=0$, how determinacy and indeterminacy of the moment sequence changes, for which times the time-dependent sequences remain moment sequences, which structure of the moment cone (i.e., the set of all moment sequences) is revealed by the time-dependent moments, and what can be said about the underlying representing measures with changing time $t$, without necessarily solving the evolution equation.

In Section \ref{sec2} we discuss in detail the case of the heat equation $\partial_t u = \nu \Delta u$, and we present some illustrative examples. \ Section \ref{sec3} is reserved for a study of the transport equation $\partial_t u = ax \cdot \nabla u$, and we give some revealing examples. \ Finally, in Section \ref{sec4} we consider an equation of the form $\partial_t u(x,t) = \nu\Delta u(x,t) + ax\cdot\nabla u(x,t)$, that is, a suitable combination of the heat and transport equations.

\section{Time-dependent moments from the heat equation $\partial_t u = \nu \Delta u$} \label{sec2}

Let $\nu>0$ and $n\in\nset$. \ Consider the \emph{heat equation}
\begin{equation}\label{eq:heat}\begin{split}
\partial_t u(x,t) &= \nu\Delta u(x,t)\\
u(x,0) &= u_0(x),
\end{split}\end{equation}
with $x \equiv (x_1,\ldots,x_n)^T \in\rset^n$ and $u_0\in\cS(\rset^n)$ a Schwartz function. \ It is well known that (\ref{eq:heat}) has the unique solution
\begin{equation}\label{eq:heatSol}
u(\,\cdot\,,t) = \Theta_{\nu,t}* u_0
\end{equation}
with the heat kernel
\begin{equation}\label{eq:heatKernel}
\Theta_{\nu t}(x) := \frac{1}{(4\pi\nu t)^{n/2}}\cdot e^{-\frac{x^2}{4\nu t}}
\end{equation}
for all $t>0$ and with $x^2 := x_1^2 + \dots + x_n^2$; here $*$ denotes convolution. \ $\Theta_{0}$ is the Dirac $\delta$-distribution since $\Theta_{0}*u_0 = u_0$ holds. \ Also, recall that the heat equation admits non-physical solutions, and when $u_0=0$, there are infinitely many solutions which grow very rapidly with $|x|\to\infty$, see \cite[Ch.\ 7]{john95}.  \ (For more on the heat equation see, e.g.,\ \cite[Ch.\ 2.3]{evans10}.)

If $u_0\in\cS(\rset^n)$, then $u(\,\cdot\,,t)\in\cS(\rset^n)$ for all $t\geq 0$. \ Hence, all moments of $u(x,t)$ exist and are time-dependent:
\[s_\alpha(t) := \int_{\rset^n} x^\alpha\cdot u(x,t)~\diff x \quad\text{with}\quad \alpha\in\nset_0^n.\]
If $u_0\geq 0$, then $u(\,\cdot\,,t)\geq 0$ for all $t\geq 0$ and therefore $s(t) := (s_\alpha(t))_{\alpha\in\nset_0^n}$ is a moment sequence with representing measure $\diff\mu = u(x,t)~\diff x$ for all $t\geq 0$. \ The one-parameter family $s(t)$ of moment sequences describes a curve in the moment cone. \ This curve is given by a polynomial (with and without the restriction $u_0\geq 0$). \ As a matter of fact, all $s_\alpha$ are polynomials in $t$, as we now prove. \ In the next lemma, $\lfloor x\rfloor$ denotes the largest $k\in\zset$ with $k\leq x$.

\begin{lem}\label{lem:moments}
Let $n\in\nset$ and $\nu>0$, and let $u$ be a solution of the heat equation (\ref{eq:heat}) with $u_0\in\cS(\rset^n)$. \ Then
\begin{enumerate}[(i)]
\item $s_\alpha \in \rset[t]$ with
\[\deg s_\alpha \leq \left\lfloor \frac{\alpha_1}{2} \right\rfloor + \dots + \left\lfloor \frac{\alpha_n}{2} \right\rfloor\]
for all $\alpha\in\nset_0^n$. \ The coefficients of $s_\alpha(t)$ depend only on the moments $s_\beta(0)$ of $u_0$ with $\beta \leq \alpha$, i.e., $\beta_j\leq \alpha_j$ for all $j=1,\dots,n$.

\item For $n=1$ and all $k\in\nset_0$ we have
\begin{align*}
s_{2k}(t) &= \sum_{j=0}^k \frac{(2k)!}{(2k-2j)!\cdot j!}\cdot s_{2k-2j}(0)\cdot (\nu t)^j
\intertext{and}
s_{2k+1}(t) &= \sum_{j=0}^k \frac{(2k+1)!}{(2k+1-2j)!\cdot j!}\cdot s_{2k+1-2j}(0)\cdot (\nu t)^j
\end{align*}
\end{enumerate}
\end{lem}
\begin{proof}
(i): \ Without loss of generality, assume that $\nu=1$. \ Since $u$ solves (\ref{eq:heat}) we have
\begin{align*}
\partial_t s_\alpha(t) \!\!&=\!\! \partial_t \int_{\rset^n} \!\!x^\alpha\cdot u(x,t)~\diff x \!\!
= \!\!\int_{\rset^n} \!\!x^\alpha\cdot \partial_t u(x,t)~\diff x \!\!
= \!\! \int_{\rset^n} \!\!x^\alpha\cdot \nu\Delta u(x,t)~\diff x,
\intertext{and since $u(\,\cdot\,,t)\in\cS(\rset^n)$ for all $t\geq 0$, partial integration (twice) gives}
&= \nu \int_{\rset^n} (\Delta x^\alpha)\cdot u(x,t)~\diff x.
\end{align*}
With $\alpha=(\alpha_1,\dots,\alpha_n)\in\nset_0^n$ this implies
\[\Delta x^\alpha = (\partial_1^2 + \dots + \partial_n^2) x^\alpha = \sum_{j=1}^n \alpha_j\cdot (\alpha_j-1)\cdot x^{\alpha - 2e_j}\]
(where $e_j$ denotes the $n$--tuple with $j$--th coordinate equal to $1$ and zeros elsewhere). \ This gives
\begin{equation}\label{eq:momInduction}
\partial_t s_\alpha(t) = \sum_{j=1}^n \alpha_j\cdot (\alpha_j-1)\cdot s_{\alpha-2e_j}(t)
\end{equation}
with initial values $s_\alpha(0)$ and with $s_{\alpha-2e_j}(t) = 0$ for all $\alpha$ with $\alpha_j \leq 1$ for some $j=1,\dots, n$. \ Now observe that (\ref{eq:momInduction}) is a recursive system of ODEs, which proves the statement.

(ii): From (\ref{eq:momInduction}) we get $\partial_t s_0(t) = 0$, i.e., $s_0(t) = s_0(0)$, and
\begin{align*}
\partial_t s_{2j}(t) &= 2j\cdot (2j-1)\cdot s_{2j-2}(t) \tag{$*$}
\intertext{as well as $\partial_t s_1(t) = 0$, i.e.\ $s_1(t) = s_1(0)$, and}
\partial_t s_{2j+1}(t) &= (2j+1)\cdot 2j\cdot s_{2j-1}(t) \tag{$**$}
\end{align*}
($*$) and ($**$) can easily be solved by recursion, thus establishing (ii).
\end{proof}

\begin{rem}
\Cref{lem:moments} shows that, at our disposal, we have an explicit polynomial, and also unique access to the moments at all times $t>0$, without the need to calculate the convolution $u(\,\cdot\,,t) = \Theta_{\nu t} * u_0$. \ In fact, only the moments $s_\alpha(0)$ of the initial data $u_0$ must be known. \ We also want to emphasize that the polynomial of the moment $s_\alpha$ only depends on the ``lower'' moments $s_\beta(0)$, i.e., $\beta_j\leq \alpha_j$ for all $j=1,\dots,n$. \ This will enable us to use these results in the case of the truncated moment problem.\exmsymbol
\end{rem}

Since the polynomials $s_\alpha(t)$ are unique and depend only on $s_\beta(0)$ with $\beta\leq \alpha$ we define them for any initial sequence $s=(s_\alpha(0))_{\alpha\in\nset_0^n}$ (irrespective of whether the sequence is a moment sequence).

\begin{dfn}\label{dfn:fp}
Let $n\in\nset$ and $d\in\nset\cup\{\infty\}$. \ For a sequence $s=(s_\alpha(0))_{\alpha\in\nset_0^n:|\alpha|\leq d}$ we define
\[\fp_s := (\fp_{s,\alpha})_{\alpha\in\nset_0^n:|\alpha|\leq d}\subset\rset[t]\]
where $\fp_{s,\alpha}$ are the polynomials $s_\alpha(t)$ as in \Cref{lem:moments}.
\end{dfn}

\begin{cor}\label{cor:pProperties}
$\fp$ has the following properties:
\begin{enumerate}[(i)]
\item $\fp_{s}(0) = s$,

\item $\fp_{\fp_{s}(t_1)}(t_2) = \fp_s(t_1 + t_2)$ for all $t_1,t_2\in\rset$, and

\item $\fp_{a\cdot s + b\cdot s'} = a\cdot \fp_s + b\cdot \fp_{s'}$ for all $a,b\in\rset$ and $s,s'$ sequences.
\end{enumerate}
\end{cor}
\begin{proof}
(i) and (iii) are clear. \ (ii) follows from the semi-group property of the heat kernel: $\Theta_{t_1} * (\Theta_{t_2}*u_0) = \Theta_{t_1 + t_2}* u_0$.
\end{proof}

\begin{exm}
For $n=1$, $k\in\nset_0$, and $\fp_{s,k}$ the polynomials in \Cref{lem:moments}(ii) we have
\begin{align}\label{eq:heatOneDimMoments}
\fp_{s,0}(t) = s_0(t) &= s_0(0)\notag\\
\fp_{s,1}(t) = s_1(t) &= s_1(0)\notag\\
\fp_{s,2}(t) = s_2(t) &= s_2(0) + 2s_0(0)\cdot t\notag\\
\fp_{s,3}(t) = s_3(t) &= s_3(0) + 6s_1(0)\cdot t\\
\fp_{s,4}(t) = s_4(t) &= s_4(0) + 12s_2(0)\cdot t + 12s_0(0)\cdot t^2\notag\\
\fp_{s,5}(t) = s_5(t) &= s_5(0) + 20s_3(0)\cdot t + 60s_1(0)\cdot t^2\notag\\
&\ \, \vdots \tag*{$\circ$}
\end{align}
\end{exm}


\begin{exm}\label{exm:fpalphas}
Let $n\in\nset$. \ By definition we have
\[s_\alpha(t) = \int_{\rset^n} x^\alpha\cdot u(x,t)~\diff x.\]
and hence as in \Cref{lem:moments} we get
\[\partial_t s_\alpha(t) = \int_{\rset^n} x^\alpha\cdot \partial_t u(x,t)~\diff x = \int_{\rset^n} x^\alpha\cdot \Delta u(x,t)~\diff x
= \int_{\rset^n} (\Delta x^{\alpha})\cdot u(x,t)~\diff x.\]
Now, for $\alpha\in\nset_0^n$ with $|\alpha|\leq 2$ and letting $e_i$ denote the $n$--tuple with $1$ on the $i$--th entry and $0$'s elsewhere, we have
\[\Delta x^\alpha = (\partial_1^2 + \dots + \partial_n^2)x^\alpha = \begin{cases} 2 & \text{for}\ \alpha= 2e_1,\ldots, 2e_n \ \text{and}\\ 0 & \text{otherwise}.\end{cases}\]
Hence
\[\partial_t s_0(t) = 0 \quad\Rightarrow\quad s_0(t)=s_0(0)\]
and
\[\partial_t s_\alpha(t) = \begin{cases} 2s_0(t) = 2s_0(0) & \text{for}\ \alpha=2e_1,\ldots, 2e_n \ \text{and}\\ 0 & \text{for}\ \alpha\in\nset_0^n,\ |\alpha|\leq 2,\ \alpha\neq 2e_1,\ldots, 2e_n .\end{cases} \tag{$*$}\]
In summary, solving the ordinary differential equations ($*$) we get
\[\fp_{s,\alpha}(t) = \!s_\alpha(t) = \!\begin{cases} \!s_\alpha(0) + 2s_0(0)\cdot t & \!\!\!\text{for }\ \!\!\alpha = 2e_1,\ldots, 2e_n \ \text{and}\\ \!s_\alpha(0) & \!\!\!\text{for }\ \!\!\alpha\in\nset_0^n,\ \!|\alpha|\leq 2,\ \!\alpha\!\neq \!2e_1,\ldots, 2e_n.\end{cases}\]
\exmsymbol
\end{exm}

With the Riesz functional $L_s$ defined by $L_s(x^\alpha) = s_\alpha$ and linearly extended to all $\rset[x_1,\dots,x_n]$ we see that $\fp_s$ describes a curve in the moment cone which is given by a polynomial in each component; that is,
 $L_{\fp_s(t)}(x^\alpha) = \fp_{s,\alpha}(t) = s_\alpha(t)\in\rset[t]$.

The heat kernel $\Theta_{\nu t}$ is itself a solution of the heat equation (\ref{eq:heat}). \ In fact, it is the fundamental solution. \ But we already mentioned that $\Theta_{0}$ is the Dirac $\delta$-distribution, i.e., the $\delta$-measure (point evaluation) at $x_0=0$. \ For $t>0$ we have that $\Theta_{\nu t}$ is a Gaussian (normal) distribution. \ For notational simplicity we set $\Theta_{0}=\delta_0$. \ Since the heat equation is invariant with respect to translation, we also know that $\Theta_{\nu t}(x-x_0)$ for any $x_0\in\rset^n$ and $t\geq 0$ are (fundamental) solutions. \ For $\delta_{x_0}$, all moments exist since it is only point evaluation at $x_0$. \ For the Gaussian $\Theta_{\nu t}(x-x_0)\in\cS(\rset^n)$, it is also true that all moments exist. \ A combination of \Cref{lem:moments} and \Cref{dfn:fp} yields the following result.

\begin{thm}\label{thm:timeaddition}
Let $k\in\nset$, $N\in\nset\cup\{\infty\}$, $p_1,\dots,p_k\in\rset^n$, $c_1,\dots,c_k\in\rset$, and $t_1,\dots,t_k\in [0,\infty)$. \ Set $\tau := \min_i t_i$ and $s := (s_\alpha)_{\alpha\in\nset_0^n:|\alpha|\leq N}$ with
\[s_\alpha := \int_{\rset^n} x^\alpha~\diff\mu(x) \qquad\text{and}\qquad \mu_0(x) := \sum_{i=0}^k c_i\cdot \Theta_{t_i}(x-p_i).\]
Then for all $t\in [-\tau,\infty)$ we have that the sequence $\fp_s(t)$ is represented by
\[\mu_t(x) := \sum_{i=0}^k c_i\cdot \Theta_{t_i+t}(x-p_i).\]
\end{thm}
\begin{proof}
By \Cref{lem:moments} and linearity of the integral (moments), in the measure $\mu_t(x)$ it is sufficient to show the statement for $k=1$. \ Now, for $t\in (-t_1,0)$, the statement follows from $\Theta_t * \Theta_{t_1-t} = \Theta_{t_1}$; for $t\in (0,\infty)$, it follows from $\Theta_t*\Theta_{t_1} = \Theta_{t_1 + t}$. \ It remains to treat the special case $t=-t_1$. \ This case follows from
\[\int_{\rset^n} x^\alpha\cdot\Theta_{t}(x-x_0)~\diff x\quad \xrightarrow{t\to 0}\quad x_0^\alpha = \int_{\rset^n} x^\alpha~\diff\delta_{x_0}(x)\]
for all $\alpha\in\nset_0^n$.
\end{proof}

\begin{rem}
Observe that in the previous result, $t<-\tau$ is not possible yet; that is, it remains an open question whether those values are accessible. \ To visualize this, take $\delta_0$, then the moments are $s_0(t)=1$, $s_1(t)=0$, and $s_2(t)=2t$. \ That is, for $t>0$ we have one Gaussian distribution but for $t<0$ we have $s_2(t) < 0$ and with $s_0(t)=1$ this implies that this moment sequence can not be represented by one point evaluation, and at least two (signed) point evaluation are required. \ Hence, additional point evaluations appear. \exmsymbol
\end{rem}

For $n\in\nset$ and $d\in\nset$ we denote by $\cS_{d,n}$ (short $\cS_d$) the set (moment cone) of all truncated moment sequences $s=(s_\alpha)_{\alpha:|\alpha|\leq d}$ in $n$ variables. \ By $\cS_{\infty,n}$ (short $\cS_\infty$) we denote the set of all moment sequences in $n$ variables, i.e., with no restrictions on $|\alpha|$ \ Also, in what follows, if the index $n$ is dropped then $n$ is any $n\in\nset$.

\begin{dfn}\label{dfn:intervall}
Let $n\in\nset$ and $d\in\nset\cup\{\infty\}$. \ For $s\in\cS_d$ we define
\[\fI_s := \{t\in\rset \,|\, \fp_s(t)\in\cS_d\}.\]
\end{dfn}

In the case $d=0$ we have $\cS_0 = \rset_{\geq 0}$ and therefore $\fp_s = (\fp_{s,\alpha})_{|\alpha|=0}$ has only the polynomial $\fp_{s,0}(t) = s_0(0)$, i.e., $\fI_s=\rset$. \ In the case $d=1$ we have $\cS_1 = \{(s_0(0),s_{e_1}(0),\dots,s_{e_n}(0))\in\rset^{n+1} \,|\, s_0 >0\}\cup\{0\}$. \ But then only the restriction $\fp_{s,0}(t) = s_0(0) > 0$ exists, which is fulfilled for all $t\in\rset$; or if $s_0(0)=0$ then we have $s_{e_1}(0)=\dots=s_{e_n}(0)=0$, i.e., $\fp_s(t)=0$ for all $t\in\rset$ and again $\fI_s=\rset$. \ Hence, $d=0$ and $d=1$ for any $n\in\nset$ are the trivial cases:
\[s\in\cS_0\ \text{or}\ \cS_1 \quad\Rightarrow\quad \fI_s = \rset.\]
Additionally, $s=0$ also implies $\fI_0=\rset$. \ In all other cases, we have the following result.

\begin{thm}\label{thm:interval}
Let $n\in\nset$, $d\geq 2$ or $d=\infty$ and $s\in\cS_d\setminus\{0\}$. \ Then
\begin{equation*}
\fI_s = [-\fd_s,\infty) \quad\text{or}\quad (-\fd_s,\infty)
\end{equation*}
with
\begin{equation*}
\fd_s \in \left[0, \frac{s_{2e_1}(0) + \dots + s_{2e_n}(0)}{2n\cdot s_0(0)}\right].
\end{equation*}
For $d=\infty$ we always have
\[\fI_s = [-\fd_s,\infty)\]
and for $d\geq 2$ (finite) we have
\[\fI_s = [-\fd_s,\infty) \quad\text{if and only if}\quad \fp_s(-\fd_s)\in\partial\cS_d\cap\cS_d.\]
\end{thm}
\begin{proof}
We prove the statement in (a) below for finite $d\geq 2$, and then in (b) below for $d=\infty$.

(a): Let $d\in\nset$ with $d\geq 2$ and $s\in\cS_d\setminus\{0\}$.

(a-i) We show that if $c\in\fI_s\neq\emptyset$ then $[c,\infty)\subseteq\fI_s$: By \Cref{dfn:intervall} we have $0\in\fI_s$, i.e., $\fI_s\neq\emptyset$. \ Let $c\in\fI_s$, then by \Cref{dfn:intervall} we have $s':=\fp_s(c)\in\cS_d$ and since $d\geq 2$ is finite by Richter's Theorem \cite[Satz 4]{richte57} there exists an at most $k = \binom{n+d}{d}$-atomic representing measure
\[\mu = \sum_{i=1}^k c_i\cdot\delta_{x_i}\]
of $s'$ with $c_i>0$ and $x_i\in\rset^n$. \ But by \Cref{thm:timeaddition} for all $t>0$ we have that $\fp_s(c+t)$ is represented by a non-negative Gaussian mixture, i.e., all moments are finite. \ Hence $\fp_s(c+t)\in\cS_d$ for all $t>0$ and $[c,\infty)\subseteq\fI_s$.

(a-ii) We show that $\fd_s:= - \inf_{t\in\fI_s} t \leq \frac{s_{2e_1}(0) + \dots + s_{2e_n}(0)}{2n\cdot s_0(0)}$: The $\fp_{s,\alpha}$ from \Cref{exm:fpalphas} show that
\begin{equation}\label{eq:onePoint}
L_{\fp_s(t)}(x_1^2 + \dots + x_n^2) = \sum_{i=1}^n \fp_{s,2e_n}(t) = \sum_{i=1}^n s_{2e_1}(0) + 2ns_0(0)\cdot t\geq 0
\end{equation}
which implies the bound on $t$ resp.\ its infimum $\fd_s$. \ By (a-i) we have $(-\fd_s,\infty)\subseteq\fI_s$.

(a-iii) ``$\fI_s=[-\fd_s,\infty)$ if and only if $\fp_s(-\fd_s)\in\cS_d$'' follows directly from (a-ii) and \Cref{dfn:intervall}.

(b) We now prove the statements for $d=\infty$.

(b-i) We show $(-\fd_s,\infty)\subseteq\fI_s$: Let $k\in\nset$, $s=(s_\alpha)_{\alpha\in\nset_0^n}\in\cS_\infty$, and denote by $s|_k \in\cS_k$ the truncated moment sequence up to order $k$, i.e., $s|_k := (s_\alpha)_{|\alpha|\leq k}$. \ Then $\fd_{s|_k}$ is a non-increasing sequence $\geq 0$. \ Hence,
\[\fd_s := \lim_{k\to\infty} \fd_{s|_k} \geq 0\]
exists and $\fd_s\leq \fd_{s|k}$ implies
\[(-\fd_s,\infty)\subseteq \bigcap_{k\in\nset} (-\fd_{s|_k},\infty) \subseteq \fI_s.\]

(b-ii) We show $-\fd_s\in\fI_s$: Since $(-\fd_s,\infty)\subseteq \fI_s$ we have that $\fp_s(-\fd_s+\varepsilon)\in\cS_\infty$ is a moment sequence for all $\varepsilon>0$. \ Hence,
\[L_{\fp_s(-\fd_s+\varepsilon)}(p)\geq 0 \tag{$*$}\]
for all $p\in\rset[x_1,\dots,x_n]$ with $p\geq 0$ and $\varepsilon>0$. \ But since ($*$) is continuous in $\varepsilon$ we have
\[L_{\fp_s(-\fd_s)}(p) = \lim_{\varepsilon\to 0} L_{\fp_s(-\fd_s+\varepsilon)}(p) \geq 0\]
for all $p\in\rset[x_1,\dots,x_n]$ with $p\geq 0$. \ Therefore, $L_{\fp_s(-\fd_s)}$ is a moment functional, $\fp_s(-\fd_s)\in\cS_\infty$ is a moment sequence, and $-\fd_s\in\fI_s$.
\end{proof}

\begin{cor}
Let $d\geq 2$ or $d=\infty$ and $s\in\cS_d\setminus\{0\}$ with $\fd_s\in\fI_s$. \ If
\[\fd_s := \frac{s_{2e_1}(0) + \dots + s_{2e_n}(0)}{2n\cdot s_0(0)},\]
then $\fp_s(-\fd_s)$ is represented by $s_0(0)\cdot\delta_0$, and if, additionally, $\fd_s>0$, then $s$ is represented by the Gaussian $s_0(0)\cdot\Theta_{\fd_s}$.
\end{cor}
\begin{proof}
Since $\fp_s(-\fd_s)\in\cS_d$ we have by (\ref{eq:onePoint})
\[L_{\fp_s(-\fd_s)}(x_1^2 + \dots + x_n^2) = 0,\]
i.e., the non-trivial representing measure of $\fp_s(-\fd_s)$ is supported only at $x=0$. \ If $\fd_s>0$, then from \Cref{thm:timeaddition} it follows that $s$ is represented by $s_0(0)\cdot\Theta_{\fd_s}$.
\end{proof}

\begin{dfn}
We call $\fd_s$ the (\emph{heat}) \emph{distance} of $s$ to the boundary of $\cS_d$.
\end{dfn}

The terminology of heat distance is due to the connection to the heat equation (\ref{eq:heat}), and to avoid confusion with the Euclidean distance in $\rset^N\supset\cS_d$ with $N = \left(\begin{smallmatrix} n+d\\ n\end{smallmatrix}\right)$.

\begin{dfn}\label{dfn:heatCurve}
Let $d\geq 2$ or $d=\infty$ and $s\in\cS_d$. \ We define
\begin{equation}\label{eq:heatCurve}
\fC_s := \fp_s(\fI_s)\subset\cS_d
\end{equation}
and we call $\fC_s$ a (\emph{heat}) \emph{curve}.
\end{dfn}

\Cref{lem:moments} and \Cref{thm:interval} imply that the \emph{heat curves} $\fC_s$ are given by polynomials and are completely contained in $\cS_d$. \ In fact one can interpret the heat curves $\fC_s$ also in the following way. \ For any sequence $s$ (not necessarily a moment sequence) we have seen that $\fp_s$ determines a curve. \ Then we ask if and when this curve $\fp_s(\rset)$ lies in the moment cone $\cS_d$. \ We get $\fC_s = \fp_s(\rset)\cap\cS_d$. \ So the heat equation generates a flow $\fp$ of sequences $s\mapsto \fp_s$ and the heat curves $\fC_s$ are the trajectories (or parts of the trajectories) which are contained in the moment cone $\cS_d$. \ While the ``future'' ($\fp_s(t)$ with $t>0$) of a moment sequence $s\in\cS_d$ always remains in $\cS_d$, the ``past'' ($t<0$) does not. \ If $s$ was an interior moment sequence, then $\fp_s$ is a well-defined path to the boundary of $\cS_d$. \ If $\fp_s(-\fd_s)\in\partial\cS_d$ also is a moment sequence (i.e.\ $\in\cS_d$), it is simpler to find a representing measure for $\fp_s(-\fd_s)$ than for $s$ (rank reduction in the Hankel matrices). \ Then, having found a (finitely atomic) representing measure of $\fp_s(-\fd_s)$ allows to construct a Gaussian mixture representation by \Cref{thm:timeaddition}.

\Cref{thm:interval} implies that two heat curves either coincide or are disjoint. \ This is stated in the following corollary.

\begin{cor}\label{cor:heatCurvesDisjoint}
Let $n\in\nset$, $d\geq 2$ or $d=\infty$. \ For $s,s'\in\cS_d$ we have either
\begin{enumerate}[(i)]
\item $\fC_s = \fC_{s'}$
\end{enumerate}
or
\begin{enumerate}[(i)]\setcounter{enumi}{1}
\item $\fC_s \cap \fC_{s'} = \emptyset$.
\end{enumerate}
\end{cor}

Hence, each heat curve $\fC_s$ is an equivalence class in the moment cone $\cS_d$.

\begin{dfn}\label{dfn:equiv}
We define $\sim_\fp$ in $\cS_d$ by:\quad $s \sim_\fp s'\ :\Leftrightarrow\ \fC_s = \fC_{s'}$.
\end{dfn}

The heat curves $\fC_s$ are the equivalence classes in $\sim$: $[s]=\fC_s$. \ For $d\geq 2$ finite, $\fI_s=[-\fd_s,\infty)$ means that the heat curve of $s$ starts at a boundary moment sequence $s'$ and every other point on $\fC_s$ is an interior moment sequence which has a unique path to the boundary moment sequence $s'$. \ That implies the following.

\begin{thm}\label{thm:pBijParametrization}
Let $n\in\nset$ and $d\in\nset\cup\{\infty\}$ with $d\geq 2$. \ Set
\[\cat_d := \bigcup_{s\in\partial\cS_d\cap\cS_d} \fC_s \subseteq\cS_d.\]
If $d$ is finite then
\begin{equation*}
\partial\cS_d\cap\cS_d\ \cong\ \cat_d/{\sim_\fp}
\end{equation*}
and
\begin{equation*}
\fp:(\partial\cS_d\cap\cS_d)\times [0,\infty)\bimap\cat_d,\ (s,t)\mapsto \fp_s(t)
\end{equation*}
is a bijective and polynomial map ($s\neq 0$) with the inverse map
\[s \mapsto (\fp_s(-\fd_s),\fd_s),\]
i.e., $\fp$ parametrizes precisely the subset $\cat_d$ of the moment cone $\cS_d$.

If $d=\infty$ we have
\begin{equation*}
\partial\cS_\infty\ \cong\ \cS_\infty/{\sim_\fp}
\end{equation*}
and for $s\neq 0$ we have the bijective and polynomial map
\begin{equation*}
\fp:\partial\cS_\infty\times [0,\infty)\bimap\cS_\infty,\ (s,t)\mapsto \fp_s(t),
\end{equation*}
i.e., $\fp$ parametrizes the moment cone $\cS_\infty$.
\end{thm}
\begin{proof}
Let $d$ be finite. \ Since by \Cref{exm:fpalphas} we have
\[\fp_{s,\alpha}(t) = s_\alpha(0) + 2s_0(0)\cdot t\]
for $\alpha=2e_1,\dots,2e_n$ with $s_0(0)\neq 0$ we find that for $s\in\partial\cS_d\cap\cS_d$ already the map
\[\fp_{s,2e_1}:[0,\infty)\to\fC_s,\ t\mapsto \fp_{s,2e_1}(t)\]
is bijective and hence $\fp_s:[0,\infty)\to\fC_s$ is bijective. \ But by \Cref{cor:heatCurvesDisjoint} two heat curves either coincide or are disjoint. \ This proves the bijectivity of $\fp$. \ The inverse map follows immediately from \Cref{thm:interval}.

For $d=\infty$ the moment cone $\cS_\infty$ is closed ($\partial\cS_\infty\cap\cS_\infty = \partial\cS_\infty$ and $\cat_\infty = \cS_\infty$).
\end{proof}

Note that $\cat_d$ is full dimensional; that is, $\cat_d$ has nonempty interior.

For $n=1$ the boundary moment sequences of the truncated moment cone can easily be described by their representing measures since the fundamental theorem of algebra holds and hence the boundary moment sequences are determined \cite{richte57}. \ We therefore study at first the $n=1$ case with $d$ finite.

Let $n=1$, $d\in\nset$, and $s=(s_{\alpha})_{\alpha=0}^{2d}$. \ Recall, the Hankel matrix of $s$ is defined by
\[\cH(s) := (s_{\alpha+\beta})_{\alpha,\beta=0}^d \in\rset^{(n+1)\times (n+1)}\]
and define
\[\cH(s) \succeq 0 \quad:\Leftrightarrow\quad v^T\cH(s)v \geq 0\ \text{for all}\ v\in\rset^{n+1},\]
and
\[\cH(s)\succ 0\quad:\Leftrightarrow\quad v^T\cH(s)v > 0\ \text{for all}\ v\in\rset^{n+1}.\]
Then we have the following result, see e.g.\ \cite[Thm.\ 9.15 and Cor.\ 9.16]{schmudMomentBook}.

\begin{lem}\label{lem:hankel1dim}
Let $d\in\nset$ and $s=(s_\alpha)_{\alpha=0}^{2d}$ a real sequence. \ The following are equivalent:
\begin{enumerate}[(i)]
\item $\cH(s)\succeq 0$ and $\cH(s)\not\succ 0$.

\item There exists a $k$-atomic measure $\mu$ with $k\leq d$ and an $a\geq 0$ such that
\begin{align*}
s_\alpha &= \int x^\alpha~\diff\mu(x)
\intertext{for all $\alpha=0,\dots,2d-1$ and}
s_{2d} &= \int x^{2d}~\diff\mu(x) + a,
\end{align*}
for some $a \ge 0$.
\end{enumerate}
If $v=(v_0,\dots,v_d)\in\rset^{d+1}$ with $v\neq 0$ and $v^T\cH(s) v=0$, then $\supp\mu\subseteq\cZ(f)$ with $f(x) = v_0 + v_1 x + \dots + v_d x^d$.
\end{lem}

In \cite{didio18gaussian} we already showed that any interior point in the moment cone has a Gaussian mixture representation. \ We will now use \Cref{thm:pBijParametrization} to show an easy way to compute it, at least for $n=1$.

\begin{thm}
Let $n=1$, $d\in\nset$, $j=0,1$, and $s=(s_\alpha)_{\alpha=0}^{2d+j}$ with $\cH(s)\succ 0$. \ For $j=1$ choose a real number $s_{2d+2}$, and for $j=0$ choose real numbers $s_{2d+1}$ and $s_{2d+2}$ such that
\[\cH(\tilde{s})\succ 0 \qquad\text{with}\qquad \tilde{s}:=(s_{\alpha})_{\alpha=0}^{2d+2}.\]
Then
\[\det(\cH(\fp_{\tilde{s}}(-t)))\]
has a smallest zero $\delta>0$ and $s$ is represented by the Gaussian mixture
\[\mu(x)=\sum_{i=1}^k c_i\cdot\Theta_{\delta}(x-x_i)\]
with $k\leq d+1$, $c_1,\dots,c_k>0$, and $x_1,\dots,x_k\in\rset$ the zeros of $f(x) = v_0 + v_1 x + \dots v_{d+1} x^{d+1}$, where $(v_0,\dots,v_{d+1})^T\in\ker\cH(\fp_{\tilde{s}}(-\delta))\setminus\{0\}$.
\end{thm}
\begin{proof}
From \Cref{lem:hankel1dim} it is easy to see that for $n=1$ a sequence $\tilde{s}=(s_\alpha)_{\alpha=0}^{2d+2}$ is a moment sequence in the interior of the moment cone if and only if $\cH(s)\succ 0$.

By \Cref{thm:interval} we have $\fp_{\tilde{s}}(-\fd_{\tilde{s}})\in\partial\cS_{2d+2}$, i.e., $\cH(\fp_{\tilde{s}}(-\fd_{\tilde{s}}))\succeq 0$ but $\not\succ 0$ and $\cH(\fp_{\tilde{s}}(t))\succ 0$ for all $t>-\fd_{\tilde{s}}$. \ Hence, $\fd_{\tilde{s}}$ is the smallest zero $>0$ of $\det(\cH(\fp_{\tilde{s}}(-t)))$.

By \Cref{lem:hankel1dim} there exists a $k$-atomic representing measure $\nu$ with $k\leq d+1$ of $(\fp_{\tilde{s},\alpha}(-\fd_s))_{\alpha=0}^{2d+1}$ with $\supp\nu\in\cZ(f)$ with $f=v_0 + \dots + v_{d+1} x^{d+1}$ and $(v_0,\dots,v_{d+1})^T\in\ker\cH(\fp_{\tilde{s}}(-\fd_{\tilde{s}}))$.

Then, by \Cref{thm:timeaddition}, $s$ is represented by the Gaussian mixture $\mu$.
\end{proof}

From \Cref{dfn:fp}, we know that each component $\fp_{s,\alpha}(t)\in\rset[t]$ depends only on $s_\beta$ with $\beta \leq \alpha$. \ Hence, in the time evolution $L_{\fp_s(t)}(p_0)$ for $p_0\in\rset[x_1,\dots,x_n]$ we find that there is a $p_t(x)\in\rset[x_1,\dots,x_n,t]$ such that
\[L_{\fp_s(t)}(p_0) = L_s(p_t)\]
for all $t\in\rset$. \ This $p_t$ can be found by rearranging
\[L_{\fp_s(t)}(p_0) = \sum_{\alpha:|\alpha|\leq \deg p_0} c_\alpha(t)\cdot s_\alpha\]
with $c_\alpha\in\rset[t]$. \ Then $p_t$ can be defined uniquely from the $c_\alpha$'s. \ Note that in the following definition and lemma, the polynomial $p_t$ (just like $\fp_s(t)$) is defined for all $t\in\rset$.

\begin{dfn}\label{dfn:pt}
Let $p_0(x) = \sum_{\alpha} c_\alpha(0)\cdot x^\alpha\in\rset[x_1,\dots,x_n]$. \ With $c_\alpha\in\rset[t]$ from
\[L_{\fp_s(t)}(p_0) = \sum_{\alpha:|\alpha|\leq \deg p_0} c_\alpha(t)\cdot s_\alpha\]
we define
\[p_t(x) := \sum_{\alpha:|\alpha|\leq \deg p_0} c_\alpha(t)\cdot x^\alpha\]
for all $t\in\rset$.
\end{dfn}

\begin{lem}\label{lem:heatPoly}
Let $p_0\in\rset[x_1,\dots,x_n]$. \ Then $p_t\in\rset[x_1,\dots,x_n,t]$ is the unique solution of
\begin{align*}
\partial_t p_t(x) &= \Delta p_t(x)
\end{align*}
for all $t\in\rset$ with initial data $p_0$. \ Additionally, we have
\[\deg p_t = \deg p_0\]
for all $t\in\rset$ where $\deg$ is the degree in $x$.
\end{lem}
\begin{proof}
Let $u_0\in C_c^\infty(\rset^n)\subset \cS(\rset^n)$ and $u$ be the unique solution of
\begin{align*}
\partial_t u(x,t) &= \Delta u(x,t)\\ u(x,0) &= u_0(x).
\end{align*}
Since $p_t\in\rset[x_1,\dots,x_n]$ for all $t\in\rset$ and $p_0\in\rset[x_1,\dots,x_n]$ is arbitrary, it is sufficient to prove $\partial_t p_t(x) = \Delta p_t(x)$ at $t=0$. \ We have from \Cref{dfn:pt}
\[\int_{\rset^n} p_t(x) \cdot u_0(x)~\diff x = \int_{\rset^n}  p_0(x) \cdot u(x,t)~\diff x.\tag{$*$}\]
By differentiating ($*$) with respect to $t$ we get
\begin{align*}
\partial_t \int_{\rset^n} p_t(x)\cdot u_0(x)~\diff x \bigg|_{t=0}
&= \partial_t \int_{\rset^n} p_0(x)\cdot u(x,t)~\diff x \bigg|_{t=0}\\
&= \int_{\rset^n} p_0(x)\cdot \partial_t u(x,t)~\diff x \bigg|_{t=0}\\
&= \int_{\rset^n} p_0(x)\cdot \Delta u(x,t)~\diff x \bigg|_{t=0}\\
&= \int_{\rset^n} (\Delta p_0(x))\cdot u(x,t)~\diff x \bigg|_{t=0}\\
&= \int_{\rset^n} (\Delta p_0(x))\cdot u_0(x)~\diff x.
\end{align*}
Since $u_0\in C_c^\infty(\rset^n)$ was arbitrary we have
\[\partial_t p_t(x) \big|_{t=0} = \Delta p_0(x)\]
which proves the statement.
\end{proof}

\begin{exm}
For $n=1$ we have
\[s_0(t) = s_0(0),\quad s_1(t) = s_1(0),\quad \text{and}\quad s_2(t) = s_2(0) + 2s_0(0)\cdot t;\]
see (\ref{eq:heatOneDimMoments}). \ For $p_0(x) = c_0(0) + c_1(0)\cdot x + c_2(0)\cdot x^2$ we therefore have
\begin{align*}
L_{s(t)}(p_0)
&= c_0(0)\cdot s_0(t) + c_1(0)\cdot s_1(t) + c_2(0)\cdot s_2(t)\\
&= c_0(0)\cdot s_0(0) + c_1(0)\cdot s_1(0) + c_2(0)\cdot [s_2(0)+ 2s_0(0)\cdot t]\\
&= [c_0(0) + 2c_2(0)\cdot t]\cdot s_0(0) + c_1(0)\cdot s_1(0) + c_2(0)\cdot s_2(0)\\
&= c_0(t)\cdot s_0(0) + c_1(t)\cdot s_1(0) + c_2(t)\cdot s_2(0)\\
&= L_{s}(p_t),
\intertext{i.e.,}
p_t(x) &= [c_0(0) + 2c_2(0)\cdot t] + c_1(0)\cdot x + c_2(0)\cdot x^2,\\
\partial_t p_t(x) &= 2c_2(0),
\intertext{and}
\partial_x^2 p_t(x) &= 2c_2(0)
\end{align*}
for all $t\in\rset$.\exmsymbol
\end{exm}

\begin{rem}
\Cref{lem:heatPoly} can be also interpreted as follows. \ The unique solution of the heat equation is gained by convolution with $\Theta_t$ and we have the well-known relation
\begin{align*}
\int_{\rset^n} g(x)\cdot (\Theta_t*f)(x)~\diff x &= \int_{\rset^n} \int_{\rset^n} g(x) \cdot \Theta_t(x-y)\cdot f(y)~\diff x~\diff y\\
&= \int_{\rset^n} (\Theta_t*g)(y)\cdot f(y)~\diff y.
\end{align*}
For a measure $\mu_0$ we define $\mu_t = \Theta_t*\mu_0$ in the same manner
\begin{subequations}\label{eq:alter}\begin{align}
\int_{\rset^n} f(x)~\diff\mu_t(x) &= \int_{\rset^n} \int_{\rset^n} f(x)\cdot \Theta_t(x-y)~\diff\mu_0(y)
\intertext{which results in}
&= \int_{\rset^n} (\Theta_t*f)(y)~\diff\mu_0(y)
\end{align}\end{subequations}
and $\mu_t$ solves the heat equation. \ Together, (\ref{eq:alter}a) and (\ref{eq:alter}b) can be used to provide an alternative proof of \Cref{lem:heatPoly}. \ This requires, however, a proof that $\Theta_t* p\in\rset[x_1,\dots,x_n,t]$ for all $p\in\rset[x_1,\dots,x_n]$. \ This can be done, but it is nontrivial. \ We believe our proof of \Cref{lem:heatPoly} is conceptually stronger, and yields a polynomial in a straightforward way.  \exmsymbol
\end{rem}

Let $n\in\nset$ and $s\in\cS_\infty$ be a moment sequence. \ The sequence $s$ is called \emph{determinate} if it has exactly one representing measure. \ Otherwise, it is called \emph{indeterminate}. \ We will now see how determinacy and indeterminacy are preserved along the heat curve $\fC_s$.

\begin{thm}\label{thm:indeterminate}
Let $n\in\nset$ and $s\in\cS_\infty$ be an indeterminate moment sequence. \ Then $\fp_s(t)$ is indeterminate for all $t\in [0,\infty)$.
\end{thm}
\begin{proof}
First we prove that $\fp_s(t)$ is indeterminate for all $t\in [0,\varepsilon)$ for some $\varepsilon>0$:

Since $s$ is indeterminate, and it has at least two distinct representing measures $\mu_0$ and $\tilde{\mu}_0$. \ Since $\mu_0$ and $\tilde{\mu}_0$ are distinct there exists a measurable $A\subset\rset^n$ such that
\[\int_{\rset^n} \chi_A(x)~\diff\mu_0(x) \neq \int_{\rset^n} \chi_A(x)~\diff\tilde{\mu}_0\tag{\#}\]
where $\chi_A$ is the characteristic function of $A$. \ Without loss of generality, $A$ is compact. \ For the time-dependent measures $\mu_t$ and $\tilde{\mu}_t$ we find from (\ref{eq:alter}) that
\begin{align*}
\int_{\rset^n} \chi_A(x)~\diff\mu_t(x) &= \int_{\rset^n} (\Theta_t*\chi_A)~\diff\mu_0(x) \tag{$*$}\\
\int_{\rset^n} \chi_A(x)~\diff\tilde{\mu}_t(x) &= \int_{\rset^n} (\Theta_t*\chi_A)~\diff\tilde{\mu}_0(x)\tag{$**$}
\end{align*}
Both ($*$) and ($**$) depend continuously on $t\geq 0$ and since for $t=0$ we have (\#) there exists an $\varepsilon>0$ such that ($*$) $\neq$ ($**$) for all $t\in [0,\varepsilon)$, i.e., $\mu_t$ and $\tilde{\mu}_t$ are two distinct representing measures of $\fp_s(t)$ and hence $\fp_s(t)$ is indeterminate for all $t\in [0,\varepsilon)$.

Now we show that for $t=\varepsilon/2$ there are $C^\infty$-functions $f_{\varepsilon/2}$ and $\tilde{f}_{\varepsilon/2}$ such that
\[\diff\mu_{\varepsilon/2}(x) = f_{\varepsilon/2}(x)~\diff x \qquad\text{and}\qquad \diff\tilde{\mu}_{\varepsilon/2}(x) = \tilde{f}_{\varepsilon/2}(x)~\diff x.\]
It is sufficient to show this for $\mu_0$:

Since $s_0 = \int_{\rset^n} 1~\diff\mu_0 < \infty$ we have $\mu_0(A)<\infty$ for all Borel-measurable sets $A\in\cB(\rset^n)$. \ Let $\nu := e^{-x^2}\cdot\lambda$, $\lambda$ the Lebesgue measure on $\rset^n$. \ Then $\mu_0$ and $\nu$ are finite measures. \ Hence, by the Lebesgue decomposition \cite[Thm.\ 3.2.3]{bogachevMeasureTheory} there exists a $\nu$-integrable function $g$ such that
\[\mu_0 = g\cdot\nu + \rho \qquad\text{with}\qquad \rho \perp \nu,\]
i.e., $\rho$ is singular with respect to $\nu$ (there exists $A\in\cB(\rset^n)$ with $\nu(A) = 0$ but $\rho(A)>0$).

We show that $\Theta_t*\rho$ for $t>0$ is no longer singular with respect to $\nu$: Let $A\in\cB(\rset^n)$ with $\nu(A) = 0$, i.e., also the Lebesgue measure $\lambda(A) = 0$. \ Then
\[\chi_A (x) := \begin{cases} 1 & \text{for}\ x\in A,\\ 0 & \text{else}\end{cases} \qquad\Rightarrow\qquad \Theta_t* \chi_A = 0\ \text{for all}\ t>0,\]
and therefore
\begin{align*}
(\Theta_t*\rho)(A) &=\!\! \int_{\rset^n} \!\chi_A(x)~\diff(\Theta_t*\rho)(x)
= \!\!\int_{\rset^n} \!\!\int_{\rset^n} \!\chi_A(x)\cdot \Theta_t(x-y)~\!\diff\rho(y)~\!\diff x\\
&= \int_{\rset^n} (\Theta_t*\chi_A)(y)~\diff\rho(y)
= \int_{\rset^n} 0~\diff\rho(y) = 0.
\end{align*}
Hence, for $t=\varepsilon/4$ we have that $\Theta_{\varepsilon/4}*\rho = h\cdot\nu$ for a $\nu$-integrable function $h$. \ In summary for $t=\varepsilon/2$ we have
\begin{align*}
\mu_{\varepsilon/2} &= \Theta_{\varepsilon/2}*\mu_0\\
&= \Theta_{\varepsilon/2}*(g\cdot\nu) + \Theta_{\varepsilon/2}*\rho\\
&= [\Theta_{\varepsilon/2}*(g\cdot e^{-x^2})]\cdot\lambda + \Theta_{\varepsilon/4}*(h\cdot\nu)\\
&= [\Theta_{\varepsilon/2}*(g\cdot e^{-x^2})]\cdot\lambda + [\Theta_{\varepsilon/4}*(h\cdot e^{-x^2})]\cdot\lambda\\
&= [\Theta_{\varepsilon/2}*(g\cdot e^{-x^2}) + \Theta_{\varepsilon/4}*(h\cdot e^{-x^2})]\cdot\lambda\\
&= f_{\varepsilon/2}\cdot\lambda
\end{align*}
with $f_{\varepsilon/2}$ a $C^\infty$-function. \ In the same way we get $\tilde{\mu}_{\varepsilon/2} = \tilde{f}_{\varepsilon/2}~\cdot\lambda$ for a $C^\infty$-function $\tilde{f}_{\varepsilon/2}$. \ We already showed that $\mu_t \neq \tilde{\mu}_t$ for all $t\in [0,\varepsilon)$, i.e., for $t= \varepsilon/2$ we get $f_{\varepsilon/2} \neq \tilde{f}_{\varepsilon/2}$.

Since the heat equation has the backwards uniqueness, see e.g.\ \cite[Ch.\ 2.3]{evans10}, we have
\[\Theta_t* f_{\varepsilon/2} \neq \Theta_t* \tilde{f}_{\varepsilon/2}\]
for all $t\geq 0$, i.e., $\mu_t \neq \tilde{\mu}_t$ for all $t\geq 0$. \ Therefore, $\fp_t(s)$ is an indeterminate moment sequence for all $t\geq 0$.
\end{proof}

\begin{exm}
Let $s := \left( e^{k^2/2} \right)_{k\in\nset_0}$. \ Then for any $c\in [-1,1]$ the sequence $s$ is represented by $\diff\mu_c = f_c~\diff x$ with
\[f_c(x) := [1+ c\cdot\sin(2\pi\ln x)]\cdot \frac{\chi_{(0,\infty)}(x)}{x\cdot \sqrt{2\pi} } e^{-(\ln x)^2/2}, \]
i.e., $s$ is an indeterminate moment sequence \cite{stielt94}. \ By the backwards uniqueness of the heat equation we have
\[\Theta_t* f_c \neq \Theta_t * f_{c'}\]
for all $c,c'\in [-1,1]$ with $c\neq c'$ and $t\geq 0$. \ Hence, $\fp_s(t)$ is indeterminate for all $t\geq 0$.\exmsymbol
\end{exm}

\begin{cor} \label{heatdeterminacy}
Let $n\in\nset$ and $s\in\cS_\infty$ be a determinate moment sequence. \ Then $\fp_s(t)$ is a determinate moment sequence for all $t\in \fI_s\cap (-\infty,0]$.
\end{cor}
\begin{proof}
Assume $\fp_s(t)$ is an indeterminate moment sequence for some $t<0$. \ Then by \Cref{thm:indeterminate} we also know that $s = \fp_0(s)$ is indeterminate, a contradiction.
\end{proof}

In \Cref{thm:indeterminate} we also proved the following result.

\begin{thm}\label{thm:smoothrepr}
Let $n\in\nset$, $s\in\cS_\infty$, and $\lambda$ be the Lebesgue measure on $\rset^n$. \ Then there exists a family $\{f_t\}_{t>0}$ of $\lambda$-integrable $C^\infty$-functions with
\[f_{t_1 + t_2} = \Theta_{t_1} * f_{t_2} = \Theta_{t_2}* f_{t_1}\]
for all $t_1,t_2>0$ such that $\fp_s(t)$ is represented by $f_t\cdot\lambda$ for all $t>0$, i.e.,
\[\fp_{s,\alpha}(t) = \int_{\rset^n} x^\alpha\cdot f_t(x)~\diff x \qquad\text{for all}\ \alpha\in\nset_0^n\ \text{and}\ t>0.\]
\end{thm}

\begin{cor} \label{cor227}
Let $n\in\nset$, $s\in\cS_\infty$ be a moment sequence and $\lambda$ be the Lebesgue measure on $\rset^n$. \ If $\fd_s > 0$ and $s$ has a representing measure $\mu$ not of the form $f\cdot\lambda$ with a $C^\infty$-function $f$, then $s$ is indeterminate.
\end{cor}
\begin{proof}
Since $\fd_s>0$ we have that $\fp_t(s)\in\cS_\infty$ is a moment sequence for $t = -\fd_s/2$. \ By \Cref{thm:smoothrepr} $s$ is represented by $f\cdot\lambda$ for a $C^\infty$-function and by assumption also by $\mu$ not of this form, i.e., $s$ has two distinct representing measures and is therefore indeterminate.
\end{proof}

For $n=1$ it can be shown that $s$ is determinate if and only if the multiplication operator $M_x$ defined by $M_x p(x) = x\cdot p(x)$ is essentially self-adjoint on $\cset[x]$ in $L^2(\rset,\mu)$. \ In higher dimensions ($n\geq 2$), both statements are no longer equivalent and several determinate definitions appear, see e.g.\ \cite[Ch.\ 14]{schmudMomentBook}.

\begin{dfn}
Let $n\in\nset$, $s\in\cS_\infty$, and $\mu$ be a representing measure of $s$. \ The moment sequence $s$ is called
\begin{enumerate}[(a)]
\item \emph{determinate} if the representing measure $\mu$ of $s$ is unique;

\item \emph{strictly determinate} if $s$ is determinate and $\cset[x_1,\dots,x_n]$ is dense in $L^2(\rset^n,\diff\mu)$;

\item \emph{strongly determinate} if $s$ is determinate and $\cset[x_1,\dots,x_n]$ is dense in $L^2(\rset^n,(1+x_i^2)~\diff\mu)$ for all $i=1,\dots,n$;

\item \emph{ultradeterminate} if $s$ is determinate and $\cset[x_1,\dots,x_n]$ is dense in $L^2(\rset^n,(1+x_1^2+\dots + x_n^2)~\diff\mu)$.
\end{enumerate}
\end{dfn}

We have the proper inclusions
\!\!\!\!\!\[\text{ultradeterminate}\ \!\!\subsetneq\ \!\!\text{strongly \!\!\! determinate}\ \!\!\subsetneq\ \!\!\text{strictly \!\!\! determinate}\ \!\!\subsetneq\ \!\!\text{determinate}.\]
It remains open how these properties are preserved along the heat curves $\fC_s$.

\section{Time-dependent moments from the transport equation $\partial_t u = ax \cdot \nabla u$} \label{sec3}

Let $a=(a_1,\dots,a_n)^T\in\rset^n$ and $ax\cdot \nabla := a_1x_1\partial_1 + \dots + a_n x_n\partial_n$. \ The transport equation
\begin{equation}\label{eq:trans}\begin{split}
\partial_t u(x,t) &= ax\cdot\nabla u(x,t)\\
u(x,0) &= u_0(x)
\end{split}\end{equation}
has the solution
\begin{equation}\label{eq:transSol}
u(x,t) = u_0(x_1 e^{a_1 t},\ldots, x_n e^{a_n t})
\end{equation}
with the moments
\begin{equation}\label{eq:transMom}
s_\alpha(t) = s_\alpha(0) \exp\left( -\sum_{i=1}^n a_i (\alpha_i +1) t \right).
\end{equation}
The moments (\ref{eq:transMom}) follow immediately from (\ref{eq:transSol}) by the transformation
\begin{equation}\label{eq:transtrans}\begin{split}
\int f(x)~\diff\mu_t(x) &= \int f(x)~\diff\mu_0(x_1 e^{a_1 t},\ldots,x_n e^{a_n t})\\
&= \int f(y_1 e^{-a_1 t},\ldots, y_n e^{-a_n t}) e^{-(a_1+\dots+a_n) t}~\diff\mu_0(y).
\end{split}\end{equation}
All results in this section are a consequence of (\ref{eq:transtrans}).

In a way entirely similar to the heat equation case, and recalling the associated polynomial $\fp$ (\Cref{dfn:fp}), we see that (\ref{eq:transMom}) can be used to define $s_\alpha(t)$ from $s_\alpha(0)$, without necessarily restricting attention to moment sequences.

\begin{dfn}\label{dfn:ft}
Let $n\in\nset$, $d\in\nset\cup\{\infty\}$, $a_1,\dots,a_n\in\rset$, and let $s=(s_\alpha(0))_{\alpha:|\alpha|\leq d}$ be any sequence. \ For all $t\in\rset$ we define $\ft_s(t) \equiv (\ft_{s,\alpha}(t))_{\alpha:|\alpha|\leq d}$ by
\begin{equation} \label{eq:trans2Mom} \begin{split}
\ft_{s,\alpha}(t) := s_\alpha(0) \exp\left( -\sum_{i=1}^n a_i (\alpha_i +1) t \right). 
\end{split}\end{equation}
\end{dfn}

The $1$--parameter family of sequences $\ft$ fulfills the following relations.

\begin{lem}\label{lem:ftProp}
For all $a,b,t_1,t_2\in\rset$ and sequences $s$ and $s'$ we have
\begin{enumerate}[(i)]
\item $\ft_{as + bs'}(t) = a\cdot\ft_s(t) + b\cdot\ft_{s'}(t)$ and

\item $\ft_{\ft_s(t_1)}(t_2) = \ft_s(t_1 + t_2)$.
\end{enumerate}
\end{lem}
\begin{proof}
Both statements follow directly from \Cref{dfn:ft}. \ (i) is a consequence of the linearity in $s$; (ii) is straightforward, after we recall the $\exp$-addition theorem $e^{a+b} = e^a e^b$.
\end{proof}

For a moment sequence (truncated or not) which is represented by an atomic measure, we get the following result.

\begin{thm}\label{thm:transMomAtomic}
Let $n\in\nset$, $d\in\nset\cup\{\infty\}$, and $s\in\cS_d$. \ If $s$ is represented by
\[\mu_0 = \sum_{i=1}^k c_i(0)\cdot \delta_{x_i(0)}\]
with $x_i(0)=(x_{i,1}(0),\ldots,x_{i,n}(0))\in\rset^n$, then $\ft_s(t)$ is represented by
\[\mu_t = \sum_{i=1}^k c_i(t)\cdot \delta_{x_i(t)}\]
with
\[c_i(t) = c_i(0) \exp\left(-\sum_{j=1}^n a_j t \right)\]
and
\[x_{i,j}(t) = x_{i,j}(0) e^{-a_j t} \tag{$j=1,\dots,n$}\]
for all $t\in\rset$.
\end{thm}
\begin{proof}
Straightforward from (\ref{eq:transtrans}).
\end{proof}

In a way entirely analogous to the case of the heat curve $\fC_s$ of a moment sequence $s$ (\Cref{dfn:heatCurve}), we define $\fI_s$ for the transport equation, and we also define the transport curve $\fT_s$.

\begin{dfn}
Let $n\in\nset$, $d\in\nset\cup\{\infty\}$, and $s\in\cS_d$. \ Let
\[\fI_s := \{t\in\rset \,|\, \ft_s(t)\in\cS_d\}\]
and define the \emph{transport curve} by
\[\fT_s := \ft_s(\fI_s)\subset\cS_d.\] 
\end{dfn}

From (\ref{eq:transtrans}) and \Cref{thm:transMomAtomic} we get the following results.

\begin{cor}
Let $n\in\nset$, $d\in\nset\cup\{\infty\}$, and $s = (s_\alpha)_{\alpha:|\alpha|\leq d}$ be a sequence. \ The following are equivalent:
\begin{enumerate}[(i)]
\item $s$ is a moment sequence.

\item $\ft_s(t)$ is a moment sequence for some $t\in\rset$.

\item $\ft_s(t)$ is a moment sequence for all $t\in\rset$.
\end{enumerate}
I.e., we have for the transport equation (\ref{eq:trans}) and $s\in\cS_d$
\[\fI_s = \rset \qquad\text{and}\qquad \fT_s = \ft_s(\rset).\]
\end{cor}
\begin{proof}
(iii) $\Rightarrow$ (ii) $\Rightarrow$ (i) is clear. \ It is therefore sufficient to prove (i) $\Rightarrow$ (iii).

(i) $\Rightarrow$ (iii): Let $s\in\cS_d$, $\mu$ be a representing measure of $s$, and $t\in\rset$. \ By (\ref{eq:transtrans}) we have that $\mu_t(x) := \mu(x_1 e^{a_1 t},\ldots,x_n e^{a_n t})$ is a representing measure of $\ft_s(t)$. \ Hence, $\ft_s(t)$ has a non-negative representing measure, and is therefore a moment sequence $\ft_s(t)\in\cS_d$ for all $t\in\rset$. \ Hence, $\fI_s = \rset$ and $\fT_s = \ft_s(\fI_s) = \ft_s(\rset)$.
\end{proof}

\begin{cor}
Let $n\in\nset$, $d\in\nset$ finite, and $s\in\cS_d$. \ Then the Carath\'eodory number is constant on $\fT_s$, i.e., all moment sequences $s'\in \fT_s$ have the same Carath\'eodory number.
\end{cor}
\begin{proof}
Denote by $\cat(s)$ the Carath\'eodory number of $s$ and by $\cat(s')$ the Carath\'eodory number of $s' := \ft_s(t)\in\fT_s$ for some $t\in\rset$. \ Let
\[\mu_0 = \sum_{i=1}^{\cat(s)} c_i\cdot\delta_{x_i} \qquad\text{and}\qquad \mu'_0 = \sum_{i=1}^{\cat(s')} c_i'\cdot\delta_{x_i'}\]
be finitely atomic representing measures for $s$ and $s'$, where the minimal number of atoms (i.e., Carath\'eodory number) is attained. \ Let
\[\mu_t = \sum_{i=1}^{\cat(s)} c_i(t)\cdot\delta_{x_i(t)} \qquad\text{and}\qquad \mu_{-t}' = \sum_{i=1}^{\cat(s')} c_i'(-t)\cdot\delta_{x_i'(-t)}\]
be the time-evolved finitely atomic measures obtained from \Cref{thm:transMomAtomic}. \ Then $\mu_t$ is a representing measure of $s'$, i.e., $\cat(s')\leq \cat(s)$, and $\mu_{-t}'$ is a representing measure of $s$, i.e., $\cat(s)\leq \cat(s')$. \ In summary, we have $\cat(s) = \cat(s')$ and since $s'\in\fT_s$ was arbitrary equality holds for all $s'\in\fT_s$.
\end{proof}

\begin{cor} \label{transportdeterminacy}
Let $n\in\nset$ and $s\in\cS_\infty$.
\emph{(Determinacy (resp. strict determinacy, strong determinacy, ultradeterminacy)} \ The following statements are equivalent.
\begin{enumerate}[{(}i)]
\item $s$ is determinate (resp. strictly determinate, strongly determinate, ultradeterminate)
\item There exists a $s'\in\fT_s$ which is determinate (resp. strictly determinate, strongly determinate, ultradeterminate).
\item All $s'\in\fT_s$ are determinate (resp. strictly determinate, strongly determinate, ultradeterminate).
\end{enumerate}
That is, determinacy, strict determinacy, strong determinacy, and ultradeterminacy are preserved on the transport curves $\fT_s$.
\end{cor}
\begin{proof}
All statements follow immediately from (\ref{eq:transtrans}). \ We will show this for strong determinacy and $(i) \Rightarrow (iii)$. \ The other statements follow with the same arguments and similar calculations as in ($*$), see below.

$(i) \Rightarrow (iii)$: Let $s$ be strongly determinate, i.e., $\cset[x_1,\dots,x_n]$ is dense in $L^2(\rset^n,(1+x_i^2)~\diff\mu_0(x))$ for all $i=1,\dots,n$ (where $\mu_0$ is the representing measure of $s$), and let $s'\in\fT_s$, i.e., there exists $t\in\rset$ such that $s' = \ft_s(t)$.

Let $A\subset\rset^n$ be a compact set. \ Since $s$ is strongly determinate there exists a sequence $(\tilde{p}_k)_{k\in\nset}\subset\cset[x_1,\dots,x_n]$ such that
\[\tilde{p}_k\xrightarrow{k\to\infty}\chi_A(x_1 e^{-a_1 t},\ldots,x_n e^{-a_n t})\]
in $L^2(\rset^n,(1+x_i^2)~\diff\mu_0(x))$ for all $i=1,\dots,n$. \ Set
\[p_k(x_1,\dots,x_n) := \tilde{p}_k(x_1 e^{a_1 t},\ldots, x_n e^{a_n t}) \in\cset[x_1,\dots,x_n].\]
Then
\begin{align*}
0\overset{\phantom{\text{(\ref{eq:transtrans})}}}{\leq} &\int |p_k(x) - \chi_A(x)|^2\cdot (1+x_i^2)~\diff\mu_t(x)\\
\overset{\phantom{\text{(\ref{eq:transtrans})}}}{=}& \int |p_k(x) - \chi_A(x)|^2\cdot (1+x_i^2)~\diff\mu_0(x_1 e^{a_1 t},\ldots,x_n e^{a_n t})\\
\overset{\text{(\ref{eq:transtrans})}}{=}& \int \left|p_k(y_1 e^{-a_1 t},\ldots,y_n e^{-a_n t}) - \chi_A(y_1 e^{-a_1 t},\ldots,y_n e^{-a_n t})\right|^2\\
&\qquad \times (1+y_i^2 e^{-2a_i t})\cdot e^{-(a_1+\dots+a_n)t}~\diff\mu_0(y_1,\dots,y_n)\tag{$*$}\\
\overset{\phantom{\text{(\ref{eq:transtrans})}}}{\leq} &\!\!\!\int \left|\tilde{p}_k(y_1,\dots,y_n) \!-\! \chi_A(y_1 e^{-a_1 t},\ldots,y_n e^{-a_n t})\right|^2 (1+y_i^2) ~\diff\mu_0(y_1,\dots,y_n)\\
&\qquad\times e^{-(a_1+\dots+a_n)t} \cdot \underbrace{\sup_{y_i\in\rset} \frac{1+y_i^2 e^{-2a_i t}}{1+y_i^2}}_{<\infty},
\end{align*}
which converges to $0$ as $k\to\infty$. \ Hence, $\cset[x_1,\dots,x_n]$ approximates the characteristic function $\chi_A$ for any compact $A\subset\rset^n$ in $L^2(\rset^n,(1+x_i^2)~\diff\mu_t(x))$ for all $i=1,\dots,n$; i.e., we have proved (3), since $s' = \ft_s(t) \ (t\in\rset)$, was arbitrary.
\end{proof}

Since the transport operator $ax\cdot\nabla$ only induces a deformation of $\rset^n$ by (\ref{eq:transtrans}), the support of any measure $\mu$ is deformed in the same way. \ Additionally, under this deformation polynomials remain polynomials and even sums of squares remain sums of squares. \ We have the following consequence.

\begin{cor}
Let $n\in\nset$, $d\in\nset\cup\{\infty\}$, and $s\in\cS_{2d}$.
\begin{enumerate}[(i)]
\item If there is a
\[p(x) = \sum_{\alpha} c_\alpha\cdot x^\alpha\in\rset[x_1,\dots,x_n]_{\leq d}\]
such that $L_s(p^2)=0$, then
\[p_t(x) := \sum_\alpha c_\alpha\cdot \exp\left(\sum_{i=1}^n a_i \alpha_i t \right)\cdot x^\alpha\in\rset[x_1,\dots,x_n]_{\leq d}\]
fulfills $L_{\ft_s(t)}(p_t^2)=0$ for all $t\in\rset$.

\item If there is a
\[p(x) = \sum_{\alpha} c_\alpha\cdot x^\alpha\in\rset[x_1,\dots,x_n]_{\leq 2d}\]
such that $p\geq 0$ and $L_s(p^2)=0$, then
\[p_t(x) := \sum_\alpha c_\alpha\cdot \exp\left(\sum_{i=1}^n a_i \alpha_i t \right)\cdot x^\alpha \geq 0\]
fulfills $L_{\ft_s(t)}(p_t)=0$ for all $t\in\rset$.
\end{enumerate}
\end{cor}
\begin{proof}
This follows immediately from (\ref{eq:transtrans}), since in (i) and (ii) we have $p_t(x) = p(x_1 e^{a_i t},\ldots,x_n e^{a_n t})\in\rset[x_1,\dots,x_n]$ for all $t\in\rset$.
\end{proof}

\section{Concluding Remarks} \label{sec4}

Observe that the operators $\Delta$ and $ax\cdot\nabla$ do not commute, and hence the solution of the partial differential equation
\begin{equation}\label{eq:comb}\begin{split}
\partial_t u(x,t) &= \nu\Delta u(x,t) + ax\cdot\nabla u(x,t)\\
u(x,0) &= u_0(x)
\end{split}\end{equation}
is neither
\[v(x,t) = (\Theta_{\nu t}*u_0)(x_1 e^{a_1 t},\ldots,x_n e^{a_n t})\]
nor
\[w(x,t) = \Theta_{\nu t}*[u_0(x_1 e^{a_1 t},\ldots,x_n e^{a_n t})].\]
The exact solution $u$ of (\ref{eq:comb}) can be \textit{implicitly} written down, but explicit expressions are not possible in general. \ However, the \textit{moments} of the solution $u(x,t)$ can be explicitly expressed by the moments of the initial data $u_0$.

\begin{lem}
Let $n\in\nset$, $a=(a_1,\dots,a_n)\in\rset^n$, $\nu > 0$, and $u_0\in\cS(\rset^n)$. \ Then the moments of the unique solution $u(x,t)$ of (\ref{eq:comb}) are recursively determined by
\begin{align*}
s_\alpha(t) &= s_\alpha(0)\exp\left(-\sum_{j=1}^n a_j (\alpha_j+1) t\right)\\
&\quad + \nu\int_0^t [\alpha_1(\alpha_1-1)\cdot s_{\alpha-2e_1}(\tau) + \ldots + \alpha_n (\alpha_n-1)\cdot s_{\alpha-2e_n}(\tau)]\\
&\qquad\qquad\times \exp\left(\sum_{j=1}^n a_j (\alpha_j+1) \tau \right)~\diff\tau \cdot \exp\left(-\sum_{j=1}^n a_j (\alpha_j+1) t \right)
\end{align*}
\end{lem}
\begin{proof}
We have
\begin{align*}
\partial_t s_\alpha(t) &= \int_{\rset^n} x^\alpha\cdot [\nu\Delta u(x,t) + ax\cdot\nabla u(x,t)]~\diff x\\
&= \int_{\rset^n} \nu\cdot [\beta_1 (\beta_1-1) x^{\beta-2e_1} + \dots + \beta_n (\beta_n-1) x^{\beta-2e_n}]\cdot u(x,t)\\
&\qquad\quad + [a_1 (\alpha_1 + 1) + \ldots + a_n (\alpha_n + 1)]\cdot x^\alpha\cdot u(x,t)~\diff x\\
&= \nu\cdot\alpha_1(\alpha_1-1)\cdot s_{\alpha-2e_1}(t) + \dots + \nu\cdot\alpha_n (\alpha_n-1)\cdot s_{\alpha_n-2e_n}(t)\\
&\qquad\quad + [a_1 (\alpha_1 + 1) + \ldots + a_n (\alpha_n + 1)]\cdot s_\alpha(t),
\end{align*}
which is an inhomogeneous ordinary differential equation. \ This proves the statement.
\end{proof}

Observe that the moments $s_\alpha(t)$ only depend on the initial moments $s_\beta(0)$ with $\beta\leq\alpha$, because of the simple structure of the differential operator $\nu\Delta + ax\cdot\nabla$.

\begin{exm}\label{exm:aPlus}
For $n=\nu=a=1$ we have
\[s_\alpha(t) = s_\alpha(0)\cdot e^{-(\alpha+1) t} + \int_0^t \alpha(\alpha-1)\cdot s_{\alpha-2}(\tau)\cdot e^{(\alpha+1) \tau}~\diff\tau\cdot e^{-(\alpha+1) t},\]
i.e.,
\begin{align*}
s_0(t) &= s_0(0) e^{-t}\\
s_1(t) &= s_1(0) e^{-2t}\\
s_2(t) &= (s_2(0)-s_0(0)) e^{-3t} + s_0(0) e^{-t}\\
s_3(t) &= (s_3(0)-3s_1(0)) e^{-4t} + 3s_1(0) e^{-2t}\\
&\ \, \vdots \tag*{$\circ$}
\end{align*}
\end{exm}

\begin{exm}\label{exm:aMinus}
For $n=\nu=1$ and $a=-1$ we have
\[s_\alpha(t) = s_\alpha(0)\cdot e^{(\alpha+1) t} + \int_0^t \alpha(\alpha-1)\cdot s_{\alpha-2}(\tau)\cdot e^{-(\alpha+1) \tau}~\diff\tau\cdot e^{(\alpha+1) t},\]
i.e.
\begin{align*}
s_0(t) &= s_0(0) e^{t}\\
s_1(t) &= s_1(0) e^{2t}\\
s_2(t) &= (s_2(0)+s_0(0)) e^{3t} - s_0(0) e^{t}\\
s_3(t) &= (s_3(0)+3s_1(0)) e^{4t} - 3s_1(0) e^{2t}\\
&\ \, \vdots \tag*{$\circ$}
\end{align*}
\end{exm}

In (\ref{eq:heatOneDimMoments}) we have written down the moments for the one-dimensional heat equation ($n= \nu = 1$ and $a=0$), and in (\ref{eq:trans2Mom}) the moments for the transport equation ($\nu = 0$). \ By comparing (\ref{eq:heatOneDimMoments}) and (\ref{eq:trans2Mom}) with Examples \ref{exm:aPlus} and \ref{exm:aMinus}, especially by looking at the case $t\to\infty$, we find that the transport operator $ax\cdot\nabla$ governs the long term behavior of the moments.

The matter of determinacy (in its various degrees) for the combination (\ref{eq:comb}) is nontrivial. \ Although we know the situation well for both the heat and the transport equations (Theorem \ref{thm:indeterminate}, Corollary \ref{heatdeterminacy} and Corollary \ref{cor227}, and Corollary \ref{transportdeterminacy}, respectively), the equations for the combined moments indicate that there are very subtle considerations to be taken into account when analyzing determinacy for (\ref{eq:comb}). \ We plan to pursue this matter in future research.

\bigskip

\providecommand{\bysame}{\leavevmode\hbox to3em{\hrulefill}\thinspace}
\providecommand{\MR}{\relax\ifhmode\unskip\space\fi MR }
\providecommand{\MRhref}[2]{%
  \href{http://www.ams.org/mathscinet-getitem?mr=#1}{#2}
}
\providecommand{\href}[2]{#2}


\end{document}